\documentclass[a4paper, 11pt]{article}

\usepackage {amsthm}
\usepackage {amsmath}
\usepackage{amsfonts}
\usepackage{amssymb}
\usepackage{bbm}
\usepackage{yfonts}
\usepackage{mathtools}

\newcommand{\x}{\times}
\newcommand{\p}{\partial}
\newcommand{\e}{\mathrm{e}}

\renewcommand{\d}{\mathrm{d}}
\newcommand{\D}{\mathrm{D}}
\newcommand{\Lap}{\mathrm{\Delta}}
\newcommand{\eqnb}{\begin{equation}}
\newcommand{\eqnbs}{\begin{equation*}}
\newcommand{\eqnes}{\end{equation*}}
\newcommand{\eqne}{\end{equation}}

\newcommand{\re}[1]{(\ref{#1})}
\newcommand{\comment}[1]{}
\newcommand{\oneChar}{\hspace{11pt}}

\newcommand{\RR}{\mathbb{R}}
\newcommand{\TT}{\mathbb{T}}
\newcommand{\PP}{\mathbb{P}}
\newcommand{\ZZ}{\mathbb{Z}}

\newcommand{\NN}{\mathbb{N}}

\newcounter{lem}
\newcounter{AppendixLemmaNumber}
\newcounter{TempCounter1}

\newtheorem{proposition}{Proposition}
\newtheorem{theorem}{Theorem}
\newtheorem{lemma}[lem]{Lemma}

\begin{document}
\title{An Eulerian-Lagrangian Form for the Euler Equations in Sobolev Spaces}
\author{
Benjamin C. Pooley \footnote{
	BCP is supported by an EPSRC Doctoral Training Award.} \footnote {B.C.Pooley@warwick.ac.uk} \footnote{Mathematics Institute, University of Warwick, Coventry, CV4 7AL, UK} \and James C. Robinson  \footnote{JCR was supported by an EPSRC Leadership Fellowship, grant EP/G007470/1.}  \footnote{J.C.Robinson@warwick.ac.uk} \footnotemark[3]}
\maketitle


\begin{abstract}
In 2000 Constantin showed that the incompressible Euler equations can be written in an ``Eulerian-Lagrangian'' form which involves the back-to-labels map (the inverse of the trajectory map for each fixed time). In the same paper a local existence result is proved in certain H\"older spaces $C^{1,\mu}$. 

We review the Eulerian-Lagrangian formulation of the equations and prove that given initial data in $H^s$ for $n\geq2$ and $s>\frac{n}{2}+1$,  a unique local-in-time solution exists on the $n$-torus that is continuous into $H^s$ and $C^1$ into $H^{s-1}$. These solutions automatically have $C^1$ trajectories. 

The proof here is direct and does not appeal to results already known about the classical formulation. Moreover, these solutions are regular enough that the classical and Eulerian-Lagrangian formulations are equivalent, therefore what we present amounts to an alternative approach to some of the standard theory.
\end{abstract}

\section{Introduction}
We study a reformulation (following Constantin \cite{Const_EL_Local_2000}) of the incompressible Euler equations on a domain ${\TT^n}:=\RR^n/2\pi\ZZ^n$ in the absence of external forcing. The Euler equations model the flow of an incompressible inviscid fluid and are (classically) formulated in terms of a divergence-free vector field $u$ (i.e. $\nabla\cdot u=0$) as follows:
\eqnb\label{eqEulerClassical}
\frac{\p u}{\p t}+(u\cdot\nabla)u+\nabla p =0
\eqne
where $p$ is a scalar potential representing internal pressure (as opposed to physical pressure at a boundary). The divergence-free condition reflects the incompressibility constraint. 

In two and particularly in three dimensions, these equations continue to be of great interest; some recent surveys include \cite{Const_Euler_2007,Gibbon_2008, Yudovich_2006}. As an illustration of the challenge posed by these equations we note that unlike the Navier--Stokes equations where global weak solutions have been known to exist since 1934 due to Leray \cite{Leray_1934}, existence of global weak solutions of the Euler equations (on periodic domains) was not proved until 2011 by Wiedemann \cite{Wiedemann_2011}, following the work of DeLellis and Sz\'ekelyhidi \cite{DeLellis_Szekelyhidi_2010}. On the spatial domain $\RR^3$, more regular local solutions ($u\in C^0([0,T];H^s)\cap C^1([0,T];H^{s-2})$ with $s>5/2$) have been known to exist since the 1970s due to Kato et al, see for example \cite{Kato_1972, Kato1974}.

In the study of the Navier--Stokes equations, results such as those found in \cite{JCR_WS_2009} motivate us to approach the classical equations of fluid mechanics from a more Lagrangian viewpoint. In that paper, Robinson and Sadowski show that if $u$ is a suitable weak solution of the Navier--Stokes equations in 3D in the sense of Caffarelli, Kohn and Nirenberg \cite{C_K_N_1982}, then almost every particle trajectory is unique and $C^1$ in time. The arguments there are based on the fact that almost all trajectories avoid the set of points $(x,t)$ where singularities could develop using the fact that the set of such points has box-counting dimension at most $5/3$. 

Constantin has studied a form for the Euler equations that involves both the classical velocity field and the so called back-to-labels map $A$ which is defined to be the inverse of the trajectory map $X$ at each time $t$. More precisely, for an evolving vector field $u$ defined on ${\TT^n}\x[0,T]$, the trajectory map solves
\eqnb\label{eqXdefn}
\left\{
\begin{array}{l}
\dfrac{\d X}{\d t}(y,t)=u(X(y,t),t)\\
\\
X(y,0)=y
\end{array}
\right.
\eqne
for each $y\in{\TT^n}$. If $u$ is divergence-free and sufficiently regular then $X$ is well defined and $X(\cdot,t)$ is bijective for each $t$. In this case we can define the back-to-labels map $A$ by setting
\eqnb\label{eqAdefn}
A(\cdot,t)\coloneqq X^{-1}(\cdot,t),
\eqne
 where we consider $X$ as a map $X(\cdot,t):{\TT^n}\rightarrow{\TT^n}$ for each $t\in[0,T]$. For the \textit{Eulerian-Lagrangian} form, as we shall continue to call it, Constantin \cite{Const_EL_Local_2000} proved local existence and uniqueness results in certain H\"older spaces on $\RR^3$ for solutions that are periodic, or satisfy suitable decay conditions.

As Yudovich \cite{Yudovich_2006} has noted, a similar combination of Eulerian and Lagrangian approaches was used to investigate the Euler equations in H\"older spaces, by G\"unther and Lichtenstein independently, as early as the 1920s (\cite{Lichtenstein_1927}, \cite{Gunther_1926}). 

First we will review the Eulerian-Lagrangian formulation and discuss how it is formally equivalent to the usual Euler equations. We then turn to the main topic of this paper which is the proof of an existence and uniqueness result for the Eulerian-Lagrangian formulation in $C^0([0,T];H^{s}(\TT^n))$ with $s>\frac{n}{2}+1$ in dimension $n\geq2$. The proof is self contained, in the sense that it neither appeals to results about the classical Euler equations, nor to the problem in H\"older spaces.   

\section{The Eulerian-Lagrangian form of the equations}
The Eulerian-Lagrangian form of the Euler equations comprises the following system:
\eqnb\label{eqELA}
\p_tA+(u\cdot\nabla)A=0,
\eqne
\eqnb\label{eqELu}
u=\PP((\nabla A)^\ast v),
\eqne
\eqnb\label{eqELv}
\p_tv+(u\cdot\nabla)v=0.
\eqne
Given an initial divergence-free velocity $u_0$ for the classical equations, we choose initial conditions for the above system as follows:
\eqnb\label{eqELICA}
A(x,0)=x,
\eqne
\eqnb\label{eqELICu}
u(x,0)=v(x,0)=u_0(x).
\eqne
We use the notation $\PP$ for the Leray projector onto the space of divergence-free functions. For a matrix $M$, $M^\ast$ denotes the transposed matrix. The vector field $v$ is called the \textit{virtual velocity} and represents the initial velocity transported by the flow. 

It will often be convenient to treat $A$ as a perturbation of the identity map on ${\TT^n}$. In this case we use the notation $\eta(x,t)\coloneqq A(x,t)-x$ and replace \re{eqELA} and \re{eqELICA} with the equations
\eqnb\label{eqELeta}
\p_t\eta +(u\cdot\nabla)\eta+u=0,\oneChar \eta(x,0)=0
\eqne
respectively. We do this because the identity map (hence $A$) does not have sufficient Sobolev regularity when considered as a function on the torus with values in $\RR^n$ (i.e. without accounting for the topology of the target torus ).

The following proposition encapsulates the derivation of \re{eqELu}  (sometimes called the Weber formula) which can be found in \cite{Const_EL_Local_2000}. 
\begin{proposition}\label{propWeber}
Let $n\geq 2$, consider $u\in C^1((0,T)\x{\TT^n})$, with $u(0)\in C^1({\TT^n})$. If $u$ is divergence-free and satisfies  \re{eqEulerClassical} for some $p$, with spatially periodic boundary conditions then $A\in C^1((0,T)\x{\TT^n};{\TT^n})$ and $u$ satisfies \re{eqELu} with $v(x,t)=u_0(A(x,t))$.
\end{proposition}
\begin{proof}
From the regularity assumptions on $u$ and periodicity of the domain we deduce that the trajectories $X(y,\cdot)\in C^2(0,T)$ and $\nabla X(y,\cdot)\in C^1(0,T)$ for all $y\in{\TT^n}$, we also have $X, \frac{\p X}{\p t}\in C^1((0,T)\x{\TT^n})$. It follows from the divergence-free condition that $\det \nabla X \equiv 1$, so $X$ is volume preserving and locally injective, hence bijective, given that ${\TT^n}$ has finite volume. By the inverse function theorem we see that $A$ exists and is an element of $C^1((0,T)\x{\TT^n})$. We now have enough regularity to make the following calculations rigorous.

From \re{eqEulerClassical} and \re{eqXdefn} we obtain
\eqnbs
\frac{\p^2X}{\p t^2}(y,t)=-\nabla p(X(y,t),t),
\eqnes
which is of course just a Lagrangian interpretation of the Euler equations.
Setting $\tilde p(y,t)=p(X(y,t),t)$ this becomes
\eqnbs
\frac{\p^2X}{\p t^2}=-((\nabla X)^\ast)^{-1}\nabla\tilde p(y,t).
\eqnes
Multiplying through by $(\nabla X)^\ast$ and changing the order of differentiation yields
\eqnb\label{propWeber:eq3}
\frac{\p}{\p t}\left[\frac{\p X_j}{\p t}\frac{\p X_j}{\p y_i}\right]=\frac{\p}{\p y_i}\left[-\tilde p +\frac{1}{2}\left|\frac{\p X}{\p t}\right|^2\right]
\eqne
for $i=1,\ldots,n$, where there is an implicit sum over $j=1,\ldots,n$ and $X_j$, $y_i$ denote the components in $\RR^n$ of $X$, $y$ respectively. Integrating \re{propWeber:eq3} in time, multiplying the corresponding vector equation by $(\nabla A)^\ast$ and evaluating at $A(x,t)$ gives
\eqnb\label{propWeber:eq4}
u(x,t)=\frac{\p X}{\p t}(A(x,t),t)=(\nabla A)^\ast u_0(A(x,t))-\nabla n
\eqne
where
\eqnbs
n(x,t)=\int_0^t \tilde p(A(x,t),s)-\frac{1}{2}\left|\frac{\p X}{\p t}(A(x,t),s)\right|^2 \d s.
\eqnes
As gradients lie in the kernel of  the Leray projector, applying $\PP$ to \re{propWeber:eq4} shows that $u$ satisfies \re{eqELu} as required. Note that $v(x,t)=u_0(A(x,t))$ satisfies \re{eqELv}, hence solutions to the  Euler equations indeed solve the Eulerian-Lagrangian form.
\end{proof}

The converse is a little more technical.
\begin{proposition}
Let $s>\frac{n}{2}+1$ and $u$, $v$, $\eta\in C^0([0,T];H^s)\cap C^1([0,T];H^{s-1})$ satisfy \re{eqELu}, \re{eqELv}, \re{eqELICu} and \re{eqELeta}. Then for some $p\in C^0([0,T];H^s)$ $u$ solves \re{eqEulerClassical}.
\end{proposition}
\begin{proof}
Since $H^{s-1}({\TT^n})\hookrightarrow L^\infty({\TT^n})$ is an algebra, we have that if $f,g\in H^{s-1}$ (scalar valued) then
\[
\p_{x_i}(fg)=(\p_{x_i}f)g + f(\p_{x_i}g)
\]
as an equlity of $L^2$ functions, for $i=1,2,\ldots, n$. Therefore, denoting the material derivative by $\D_t\coloneqq\p_t+(u\cdot\nabla)$, for $f,g\in C^0([0,T];H^{s-1})\cap C^1([0,T];H^{s-2})$ we have 
\eqnb\label{prop2:eqDistr}
\D_t (fg)= (\D_tf)g+f(\D_tg).
\eqne
Moreover, if $f\in H^s$,
\[
(u\cdot\nabla )\nabla f = \nabla((u\cdot\nabla)f)- (\nabla u)^\ast \nabla f.  
\]
Hence the classical commutation relation 
\eqnb\label{eqComtnRel}
\D_t\nabla f=\nabla \D_t f-(\nabla u)^\ast\nabla f
\eqne
holds as an equality in $L^2$, when $f\in C^0([0,T];H^{s})\cap C^1([0,T];H^{s-1})$. 

Since $u$ satisfies \re{eqELu}, we may write
\eqnb\label{eqELuAlt}
u(x,t)= v+(\nabla\eta)^\ast v-\nabla n
\eqne    
for some real-valued $n$. Then by \re{prop2:eqDistr} and \re{eqComtnRel} the following calculations are justified:
\eqnb\label{eqReverseDeriv}
\begin{aligned}
\D_t u&=\D_tv+(\D_t\nabla \eta)^\ast v+(\nabla \eta)^\ast \D_tv-\D_t\nabla n\\
&=(\nabla \D_t \eta)^\ast v - (\nabla u)^\ast(\nabla \eta)^\ast v -\nabla \D_t n +(\nabla u)^\ast \nabla n\\
&=-(\nabla u)^\ast[v+(\nabla \eta)^\ast v-\nabla n]-\nabla \D_t n\\
&=-(\nabla u)^\ast u-\nabla \D_t n\\
&=-\nabla p
\end{aligned}
\eqne
where $p=\frac{1}{2}|u|^2+\D_tn$.
\end{proof}

\section{An Existence and Uniqueness Theorem}
For $r\geq 0$, we will use the notation $H^r$ variously for scalar or vector valued functions in $H^r(\TT^n)$ (componentwise), where this does not cause ambiguity. 
We will often consider functions in spaces of the form $C^0([0,T];(H^{s}(\TT^n))^n)$. To simplify notation we define $\Sigma_s(T)$ (usually denoted $\Sigma_s$) for $T \geq 0$ and $s\geq 0$ by
\eqnbs
\Sigma_s(T):=C^0([0,T];(H^{s}(\TT^n))^n).
\eqnes
 We consider the natural norm on $\Sigma_s$:
\[
\|u\|_{\Sigma_s}=\sup_{t\in[0,T]}\|u(t)\|_{H^{s}}.
\]
The aim of the rest of this paper is to prove the following theorem.
\begin{theorem}\label{thmExistUniq}
If $n\geq 2$, $s>\frac{n}{2}+1$ and $u_0\in H^{s}$ is divergence free
then there exists $T>0$, such that the system (\ref{eqELA}--\ref{eqELv}) with initial conditions \re{eqELICA} and \re{eqELICu} has a unique solution $A,u,v$ such that $\eta,u,v\in\Sigma_{s}(T)\cap C^1([0,T];H^{s-1})$ where $\eta(x,t)=A(x,t)-x$. Moreover $A \in C^1([0,T]\x{\TT^n})$ as a map into the torus.   
\end{theorem}

 We will prove this by constructing a contracting iteration scheme using the equations \re{eqELu},\re{eqELv} and \re{eqELeta}. More precisely, given $u\in\Sigma_s(T)$ we find $v, \eta\in\Sigma_s \cap C^1([0,T]\x{\TT^n})$, solutions of
\[
\p_t \eta + (u\cdot\nabla) \eta=-u, \;\eta(0,x)=0
\]  
and
\[
\p_t v+ (u\cdot\nabla) v=0,\;v(0,x)=u_0(x).
\]
We then construct the next iterate of $u$, using
\[
u'=\PP[(\nabla A)^\ast v]
\]
and show that $u\mapsto u'$ is a contraction on a certain subset of $\Sigma_s$.

In the case of H\"older spaces, Constantin constructed an iteration scheme that was instead a contraction with respect to $A$. This involves controlling differences between candidate virtual velocities ($v_1$ and $v_2$, say) in terms of the difference between the respective back-to-labels maps ($A_1$ and $A_2$). This can be achieved, using the fact that $v_i = u_0 (A_i)$ is a solution to \re{eqELv}. In the H\"older setting this is a natural way to proceed, however, relying on this \textit{a posteriori} knowledge about the solution introduces an extra technicality when we work in Sobolev spaces. For this reason we will proceed as described above, relying only on \textit{a priori} estimates. Following the proof, we shall see how the argument differs if the contraction is with respect to $A$, in particular we get an alternative proof under the additional assumption that $s\in\ZZ$.


 We begin the proof of Theorem \ref{thmExistUniq} by stating two inequalities concerning the advection term $(u\cdot\nabla)v$, using the notation $B(u,v)\coloneqq(u\cdot\nabla)v$. Both of these results can be proved following the steps in \cite{ConstFoias,JCR_Sad_Sil_2012} (the only difference being that $B$ here does not include a Leray projection).   

\begin{lemma}\label{lemBilBndHs}
For $s>\frac{n}{2}$ there exists $C_1>0$ such that if $u\in H^s$ and $v\in H^{s+1}$ then $B(u,v)\in H^s$ and
\eqnb\label{lemBilBndHs:bound}
\|B(u,v)\|_{H^s}\leq C_1\|u\|_{H^s}\|v\|_{H^{s+1}}.
\eqne
\end{lemma}

This is really just the fact that $H^s$ is a Banach algebra. For the second lemma the assumption that $u$ is divergence-free allows us to ``save a derivative'' by means of the identities 
\[(B(u,(-\Lap)^{r/2}v),(-\Lap)^{r/2}v)_{L^2}=0
\]
 for $r\in[0,s]$.
\begin{lemma}\label{lemTrilBndHs}
 If $s>\frac{n}{2}+1$ there exists $C_2>0$ such that for $u\in H^s$, $v\in H^{s+1}$ with $u$ divergence-free we have
\eqnb\label{lemTrilBndHs:bound}
|(B(u,v),v)_{H^s}|\leq C_2\|u\|_{H^s}\|v\|^2_{H^s}.
\eqne
\end{lemma}

We use the following shorthand for closed balls in $\Sigma_{s}$:
\eqnbs
B_M = \overline{B_{\|\cdot\|_{\Sigma_s}}(0,M)},
\eqnes
i.e. $B_M$ is the closed unit ball centred at the origin of radius $M>0$ with respect to the norm $\|\cdot\|_{\Sigma_s}$. Where ambiguity could arise we write $B_M(T)$ for the closed ball in $\Sigma_s(T)$.

\begin{lemma}\label{lemW}
If $s>\frac{n}{2}+1$ and $\eta,v\in \Sigma_{s}(T)$ then $\PP[(\nabla \eta)^\ast v]\in \Sigma_{s}$ and there exists  a constant $C_3>0$ (independent of $\eta$, $v$, $t$ and $T$) such that for fixed $t$,
\eqnb\label{lemW:bound1}
\|\PP[(\nabla \eta)^\ast v]\|_{H^r}\leq C_3\|\eta\|_{H^{s}}\|v\|_{H^r},
\eqne
where $r=s$ or $r=s-1$. Furthermore, there exists $C_3'>0$ such that for any $M>0$ and $T>0$, the following bounds hold uniformly with respect to $t\in[0,T]$ for any $ \eta_1,\eta_2,v_1,v_2 \in B_M(T)$:
\eqnb\label{lemW:LipBound}
\|\PP[(\nabla \eta_1)^\ast v_1-(\nabla \eta_2)^\ast v_2]\|_{X}\leq C_3'M(\| \eta_1- \eta_2\|_{X}+ \|v_1-v_2\|_{X}).
\eqne
where $X$ is $L^2(\TT^n)$ or $H^{s-1}$. 
\end{lemma}
\begin{proof}
For continuity into $H^{s-1}$ we use the fact that $H^{s-1}$ is a Banach algebra. More precisely, we see that 
\eqnb\label{lemW:eq1}
\begin{aligned}
\|\PP[(\nabla \eta_1)^\ast v_1-(\nabla \eta_2)^\ast v_2]\|_{H^{s-1}}&\leq C\| \eta_1- \eta_2\|_{H^{s}}\|v_1+v_2\|_{H^{s-1}}\\
&\oneChar+C\|\nabla \eta_1+\nabla \eta_2\|_{H^{s-1}}\|v_1-v_2\|_{H^{s-1}},
\end{aligned}
\eqne
where $C>0$ is  independent of the $\eta_i$ and $v_i$. 
The key step in the proof of \re{lemW:bound1} when $r=s$ is that if $\eta,v\in C^2$ then for some $q\in H^{s}$, 
\eqnbs
\begin{aligned}
\p_{x_i}\PP[(\nabla \eta)^\ast v]&=\p_{x_i}(\p_{x_j}\eta_kv_k)-\p_{x_i}\p_{x_j}q\\
&=\p_{x_j}(\p_{x_i}\eta_kv_k)-\p_{x_i}\eta_k\p_{x_j}v_k+\p_{x_j}\eta_k\p_{x_i}v_k-\p_{x_i}\p_{x_j}q
\end{aligned}
\eqnes
where sums are taken implicitly over $k$. The left-hand side is already divergence-free so projecting again removes the gradient terms and yields
\eqnb\label{lemW:eq0}
\p_{x_i} \PP[(\nabla \eta)^\ast v]=\PP[(\nabla \eta)^\ast \p_{x_i} v-(\nabla v)^\ast \p_{x_i} \eta].
\eqne
By continuity, this still holds if we only have $\eta,v\in H^{s}$.
A calculation similar to \re{lemW:eq1} applied to \re{lemW:eq0} yields continuity with respect to the $H^{s}$ norm as claimed. 

The inequalities \re{lemW:bound1} for $r=s-1$ and $r=s$ are obtained by taking the $H^{s-1}$ norms of $\PP[(\nabla \eta)^\ast v]$ and \re{lemW:eq0} respectively.

To prove \re{lemW:LipBound}, we again use the fact that $\PP$ removes gradients. Indeed for $f$, $g\in H^{s}$, we have
\eqnb\label{lemW:eq2}
\PP((\nabla f)^\ast g)=\PP(\nabla(f\cdot g)-(\nabla g)^\ast f)=-\PP((\nabla g)^\ast f).
\eqne
Setting $f= \eta_1- \eta_2$, $g=v_1+ v_2$, we see that the calculations in \re{lemW:eq1} can be modified to give the required result. Note that for the $L^2$ bound we use the fact that \re{lemW:eq1} holds if we replace $H^{s}$ with $L^\infty$ and $H^{s-1}$ with $L^2$.
\end{proof}

The next lemma gives uniform bounds on the $H^s$ norms of solutions to the transport equations \re{eqELA} and \re{eqELv}. We will consider the following system:
\eqnb\label{eqTransport}
\left\{
\begin {array}{l}
\p_t f + (u\cdot\nabla) f =g\\
f(0)=f_0
\end{array}
\right.
\eqne
where $f,g:[0,T]\x{\TT^n}\to\RR^n$ and $u$ is divergence free.

\begin{lemma}\label{lemTransport}
Let $s>\frac{n}{2}+1$ and fix $f_0\in H^{s}$, $g\in\Sigma_s$. If $u\in\Sigma_s$ is non-zero and divergence free then there exists a unique solution $f$ to \re{eqTransport}. Furthermore, the solution $f\in \Sigma_{s}\cap C^1([0,T];H^{s-1})\cap C^1([0,T]\x {\TT^n})$ and  there exists $C_4>0$ (from Lemma \ref{lemTrilBndHs}) such that if $r,t\in[0,T]$ we have:
\eqnb\label{lemTransport:bound1}
\|f(t)\|_{H^{s}}\leq \left(\|f(r)\|_{H^{s}}+\frac{\|g\|_{\Sigma_s}}{C_4\|u\|_{\Sigma_s}}\right)\exp(C_4|t-r|\|u\|_{\Sigma_{s}})-\frac{\|g\|_{\Sigma_s}}{C_4\|u\|_{\Sigma_s}}.
\eqne
\end{lemma}
\begin{proof}
By the method of characteristics we obtain a solution $f\in C^1([0,T]\x{\TT^n})$. The formal argument that follows motivates our consideration of the regularity of $f$.
Taking the $H^{s}$ product of  \re{eqTransport} with $f$ yields 
\[
\frac{1}{2}\frac{\d}{\d t}\|f\|^2_{H^{s}}=-(B(u,f),f)_{H^{s}}+(f,g)_{H^s}.
\]
By Lemma \ref{lemTrilBndHs}, there exists $C>0$ such that for all $t\in[0,T]$,
\eqnb\label{lemTheta:bound2}
\frac{1}{2}\frac{\d}{\d t}\|f(t)\|_{H^{s}}^2\leq C\|u(t)\|_{H^{s}}\|f(t)\|_{H^{s}}^2+\|g(t)\|_{H^s}\|f(t)\|_{H^s}.
\eqne
Now \re{lemTransport:bound1} follows from Gronwall's inequality. In the case $r>t$, this argument is applied to the time-reversed equation, that is, using the fact that for fixed $r$, $-f(r-t)$ is transported by $-u(r-t)$ with forcing $g(r-t)$.

To properly justify this we can proceed by a Galerkin method. For each $N\in\NN$ we find a solution to the system
\eqnb\label{lemTransport:eqTrunc}
\left\{
\begin{array}{l}
\p_t f_N +P_NB(u_N,f_N)=g_N\\
f_N(r)=P_N f(r),
\end{array}
\right.
\eqne 
on $[r,T]$, where $P_N$ denotes truncation up to Fourier modes of order $N$ (in space), $u_N\coloneqq P_Nu$ and $g_N\coloneqq P_N g$.
The estimate \re{lemTransport:bound1} applies to $f_N$ so by a standard argument using the Aubin-Lions lemma we obtain a weak solution $h\in L^\infty(r,T; H^s)$ such that $\p_th\in L^\infty(r,T; H^{s-1})$, hence $h\in C^0([0,T];H^{s-1})$. Using the divergence free property we obtain uniqueness of solutions $h\in L^2(r,T;H^1)$ with time derivative $\p_t h\in L^2(r,T;L^2)$. Indeed, if $h$ and $\tilde h$ are two such solutions it follows from \re{eqTransport} that
\[
\frac{\d}{\d s}\|h-\tilde h\|^2_{L^2}=0.
\]
Therefore $f=h$, i.e. this weak solution agrees with our $C^1$ classical solution on $[r,T]$. 

We now prove \re{lemTransport:bound1} in the case $r\leq t$. Since $f_N\to f$ in $L^2(r,T; H^{s-1})$, we may choose a dense countable subset $\{t_k\}_{k=1}^\infty\subset[r,T]$ such that $f_N(t_k)\to f(t_k)$ in $H^{s-1}$ as $N\to\infty$ for each $k$. The formal argument above is valid on the truncated system, thus 
\eqnb\label{lemTransport:eqApproxIneq}
\|f_N(t_k)\|_{H^s}\leq \left(\|P_Nf(r)\|_{H^s}+\frac{\|g\|_{\Sigma_s}}{C\|u_N\|_{\Sigma_s}}\right)\exp(C|t_k-r|\|u\|_{\Sigma_s})-\frac{\|g_N\|_{\Sigma_s}}{C\|u\|_{\Sigma_s}}.
\eqne
Hence, passing to a subsequence of $f_N$ for each $k$ with a diagonalisation argument, we may assume that for all $k$, $f_N(t_k)$ converges weakly in $H^s$ as $N\to\infty$. Moreover, by the choice of the points $t_k$ and uniqueness of weak limits, we must have $f_N(t_k)\rightharpoonup f(t_k)$ in $H^s$. Taking the $\liminf$ of \re{lemTransport:eqApproxIneq} with respect to  $N\to\infty$ yields
\eqnb\label{lemTransport:eqprebound1}
\|f(t_k)\|_{H^s}\leq \left(\|f(r)\|_{H^s}+\frac{\|g\|_{\Sigma_s}}{C\|u\|_{\Sigma_s}}\right)\exp(C|t_k-r|\|u\|_{\Sigma_s})-\frac{\|g\|_{\Sigma_s}}{C\|u\|_{\Sigma_s}}.
\eqne
To prove \re{lemTransport:bound1} and the weak continuity of $f$ into $H^s$ we will use the fact that a weakly convergent sequence  in $H^{s-1}$ that is also bounded in $H^s$ must converge weakly in $H^s$ to the same limit by the Banach--Alaoglu theorem. Indeed if $x_k\rightharpoonup x$ in $H^{s-1}$ is bounded in $H^s$ then any subsequence admits a further subsequence converging weakly in $H^s$ to $x$ by the uniqueness of weak limits.   

From this, \re{lemTransport:bound1} follows by the density of $\{t_k\}$ and the continuity of $f$ into $H^{s-1}$. Indeed, in the case $t\geq r$, for any subsequence $(t_{k_\ell})_{\ell=1}^\infty\subset(t_k)_{k=1}^\infty$ such that $t_{k_\ell}\to t$ we have $f(t_{k_\ell})\rightharpoonup f(t)$ in $H^s$. Applying \re{lemTransport:eqprebound1} at $t_{k_\ell}$ and taking the $\liminf$ as $\ell\to\infty$ yeilds \re{lemTransport:bound1} at time $t$. For $t<r$ the required bounds are obtained in the same way from the time-reversed version of \re{lemTransport:eqTrunc}. 

We have shown that $\|f(t)\|_{H^s}$ in bounded uniformly, not merely almost everywhere. Therefore for any fixed $\tau\in[0,T]$ and any sequence $\{\tau_k\}\subset[0,T]$ such that $\tau_k\to\tau$ we deduce, by the continuity into $H^{s-1}$,  that $f(\tau_k)\rightharpoonup f(\tau)$ in $H^s$. This says that $f$ is weakly continuous into $H^s$.  

To see that $f\in \Sigma_{s}$ it is therefore enough to show that $\|f(t)\|_{H^{s}}$ is continuous. This is the case since for all $r,t\in[0,T]$, \re{lemTransport:bound1} gives bounds of the form
\[
(\|f(r)\|_{H^s}+\alpha)\e^{-\beta|t-r|} -\alpha\leq \|f(t)\|_{H^s}\leq(\|f(r)\|_{H^s}+\alpha)\e^{\beta|t-r|} -\alpha
\]
for time independent constants $\alpha,\beta>0$, where the first inequality comes from \re{lemTransport:bound1} with $r$ and $t$ interchanged.

The fact that $f\in C^1([0,T];H^{s-1})$ follows from the fact that $\p_tf\in \Sigma_{s-1}$ which can be seen from the regularity of the other terms in \re{eqTransport}. 
\end{proof}

\begin{lemma}\label{lemDifferencesTransport}
For $s>n/2+1$ fix $u_1$, $u_2\in\Sigma_s$ and $f_0\in H^s$. Let $g_1=g_2=0$ or $g_i=-u_i$ for $i=1,2$. If $f_1$, $f_2$ are the solutions of \re{eqTransport} corresponding to $u_1$, $u_2$, $g_1$, $g_2$ respectively, then in the case that $g_1=g_2=0$, there exists $C_5>0$ depending only on $s$ such that 
\eqnb\label{lemDifferencesTransport:bound1}
\|f_1(t)-f_2(t)\|_{L^2}\leq C_5\|f_1+f_2\|_{\Sigma_s}\|u_1-u_2\|_{\Sigma_0} t
\eqne 
for all $t\in[0,T]$. In the case that $g_i=-u_i$ for $i=1,2$ we instead have
\eqnb\label{lemDifferencesTransport:bound2}
\|f_1(t)-f_2(t)\|_{L^2}\leq (C_5\|f_1+f_2\|_{\Sigma_s}+1)\|u_1-u_2\|_{\Sigma_0} t
\eqne 
\end{lemma}
\begin{proof}
Using the anti-symmetry of $(B(u_1-u_2,\cdot),\cdot)_{L^2}$ we have, for $t\in[0,T]$,
\begin{multline*}
\frac{\d}{\d t} \|f_1-f_2\|_{L^2}^2\leq |(B(u_1-u_2,f_1+f_2),f_1-f_2)_{L^2}| + 2|(g_1-g_2,f_1-f_2)|\\
\leq C \|f_1+f_2\|_{H^s}\|u_1-u_2\|_{L^2}\|f_1-f_2\|_{L^2}+2\|g_1-g_2\|_{\Sigma_0}\|f_1-f_2\|_{L^2}\\
\leq C\|f_1+f_2\|_{\Sigma_s}\|u_1-u_2\|_{\Sigma_0}\|f_1-f_2\|_{L^2} +2\|g_1-g_2\|_{\Sigma_0}\|f_1-f_2\|_{L^2}
\end{multline*}
Where $C$ depends on the embedding $H^{s-1}\hookrightarrow L^\infty$.  Formally dividing by $\|f_1-f_2\|_{L^2}$ and integrating the resulting inequality gives \re{lemDifferencesTransport:bound1} or \re{lemDifferencesTransport:bound2} depending on the choice of $g_1$ and $g_2$. Justifying this last step is straightforward.
\end{proof}

We are now in a position to prove the main result.
\begin{proof}[Proof of Theorem \ref{thmExistUniq}.]
Fix $s>n/2+1$ and let $C_3$, $C_4$ be the constants in \re{lemW:bound1}, \re{lemTransport:bound1} (from Lemmas \ref{lemW} and \ref{lemTransport}) respectively. Fix $M>\|u_0\|_{H^s}$ and $T>0$ so that 
\[
\exp(C_4TM)\|u_0\|_{H^s}\left(\frac{C_3}{C_4}[\exp(C_4TM)-1]+1\right)\leq M.
\]
 Let $u\in B_M(T)$ be a divergence free function and let $\eta$ be the solution of \re{eqTransport} for the flow $u$ with initial data $\eta_0=0$ and forcing $g=u$. Let $v$ be the solution for initial data $v_0=u_0$ with $g=0$. Define $Su:=\PP[(\nabla \eta)^\ast v + v]$, then by Lemmas \ref{lemW} and \ref{lemTransport},
\eqnb\label{thmExistUniq:B_MBnd}
\|Su(t)\|_{H^s}\leq \exp(C_4tM)\|u_0\|_{H^s}\left(\frac{C_3}{C_4}[\exp(C_4tM)-1]+1\right)\leq M
\eqne   
for all $t\in[0,T]$. Hence $S:B_M(T)\to B_M(T)$. Note that $Su(\cdot,0)=u_0$ even if $u(\cdot,0)\neq u_0$.

 We next show that $S$ is a contraction on $B_M(T)$ in the $L^2$ norm if $T$ is sufficiently small. For $u_1$, $u_2\in B_M(T)$ we construct $v_i$ and $\eta_i$ from $u_i$ as above for $i=1,2$ with $v_1(\cdot,0)=v_2(\cdot,0)=u_0$. Now 
\eqnb\label{thmExistUniq:Contraction}
\begin{aligned}
\|Su_1-Su_2\|_{L^2}&\leq C_a\|\eta_1-\eta_2\|_{L^2}+C_b\|v_1-v_2\|_{L^2}\\
&\leq(C_c\|v_1+v_2\|_{\Sigma_s}+C_d\|\eta_1+\eta_2\|_{\Sigma_s}+C_e) T\|u_1-u_2\|_{\Sigma_0}\\
&\leq C(u_0,M,T)\|u_1-u_2\|_{\Sigma_0},
\end{aligned}
\eqne
where $C_a,\ldots,C_e$ denote various constants arising from the application of Lemmas \ref{lemW}, \ref{lemTransport} and \ref{lemDifferencesTransport}. Keeping careful track of the constants shows that $C(u_0,M,T)$ is given by the formula
\eqnb\label{thmExistUnique:Const}
\begin{aligned}
C(u_0,M,T)&\coloneqq2T\left[\left(C_5(C_3'M+1)\|u_0\|_{H^s}+\frac{C_3'C_5M}{C_4}\right)\exp(C_4TM)\right.\\
&\left.\oneChar+C_3'M\left(\frac{1}{2}-\frac{C_5}{C_4}\right)\right]
\end{aligned}
\eqne
Where $C_3'$, $C_4$, $C_5$ are the constants from Lemmas \ref{lemW}, \ref{lemTransport} and \ref{lemDifferencesTransport} respectively.
  Taking the supremum of \re{thmExistUniq:Contraction} with respect to $t$ and choosing $T>0$ small enough, we see that $S$ is a contraction in the required sense.

We conclude that $S$ has a unique accumulation point $u$, in the closure of $B_M$ with respect to $\|\cdot\|_{\Sigma_0}$. Since $B_M(T)$ is convex and closed in $\Sigma_s$ it is weakly closed, hence $u\in B_M(T)$ is a fixed point of $S$. A fixed point of $S$, along with associated back-to-labels map and virtual velocity, clearly give a  solution to the Eulerian-Lagrangian formulation of the Euler equations with the required regularity. The contraction argument gives uniqueness in $B_M(T)$ and it remains to prove that we have uniqueness in $\Sigma_s(T)$.

Since $S$ is a contraction on $B_M(\widetilde T)$ for any $\widetilde T\in(0,T]$, we have by continuity of $\|u(t)\|_{H^s}$, that if $u'$, $A'$ and $v'$ also satisfy (\ref{eqELA}--\ref{eqELv}) with $u'\in\Sigma_s(T)$, then $u(t)=u'(t)$ when $0\leq t\leq\min(T,\inf\{r: \|u'(r)\|_{H^s}=M\})$. 

Now we know that for all $k\in\NN$ there exists $T_k\leq T$ such that $S$ is a contraction on $B_{M+1/k}(T_k)$ and we may assume $T_k\to T$ as $k\to\infty$. By the previous observation, this means that $u$ is the unique solution in $\Sigma_s(T-\varepsilon)$ for all $\varepsilon>0$, hence by continuity $u$ is the unique solution in $\Sigma_s$ as required.  

The proof that $u\in C^1([0,T];H^{s-1})$ uses the same trick as Lemma \ref{lemW} to save a spatial derivative (we have only shown that $\nabla \eta_t\in H^{s-2}$, which might otherwise limit the regularity of $u$). 
By definition $u=\PP[(\nabla \eta)^\ast v+v]$. We use \re{lemW:eq2} from the proof of Lemma \ref{lemW}. Precisely we have
\eqnbs
\begin{aligned}
&\frac{1}{h}\left\|u(t+h)-u(t)-h\PP[(\nabla \eta(t))^\ast\p_tv(t)+\p_tv(t)+(\nabla v(t))^\ast\p_t \eta(t)]\right\|_{H^{s-1}}\\
&\oneChar\leq\frac{1}{2h}\left\|\PP[(\nabla \eta(t+h)+\nabla \eta(t))^\ast(v(t+h)-v(t)- h\p_tv)]\right\|_{H^{s-1}}\\
&\oneChar\oneChar+\frac{1}{2h}\left\|\PP[(\nabla v(t+h)+\nabla v(t))^\ast(\eta(t+h)-\eta(t)- h\p_t\eta)]\right\|_{H^{s-1}}\\
&\oneChar\oneChar+\frac{1}{2}\|\PP[(\nabla \eta(t+h)-\nabla \eta(t))^\ast\p_t v(t)]\|_{H^{s-1}}\\
&\oneChar\oneChar+\frac{1}{2}\|\PP[(\nabla v(t+h)-\nabla v(t))^\ast\p_t \eta(t)]\|_{H^{s-1}}\\
&\oneChar\oneChar+\frac{1}{h}\|v(t+h)-v(t)-h\p_tv(t)\|_{H^{s-1}}.
\end{aligned}
\eqnes
Since $H^{s-1}$ is an algebra and $\eta, v\in C^0([0,T];H^{s})\cap C^1([0,T];H^{s-1})$, the right-hand side vanishes as $h\to 0$. Therefore $u\in C^1([0,T];H^{s-1})$ and 
\[
\p_t u= \PP[(\nabla \eta(t))^\ast\p_tv(t)+\p_tv(t)(\nabla v(t))^\ast\p_t \eta(t)].
\]
\end{proof}

\section{An Alternative Iteration}
Here we exhibit an alternative proof of existence and uniqueness for (\ref{eqELA}--\ref{eqELv}), which is based on contractions with respect to $A$ rather than $u$. The extra technicality in this approach is contained in the following lemma, which is proved in an appendix. We will denote the identity map on ${\TT^n}$ by $\iota$ and use the correspondence between maps $\TT^n\to\RR^n$ and $\TT^n\to\TT^n$ without comment.

\setcounter{AppendixLemmaNumber}{\value{lem}}
\begin{lemma}\label{lemComposition}
Let $s\in\ZZ$ with $s>\frac{n}{2}+1$ and fix $f, g\in H^s$. If $g+\iota$ is a volume preserving map then $f\circ (g+\iota)\in H^s$ and
\eqnb\label{lemComposition:bound1}
\|f\circ (g+\iota)\|_{H^s}\leq C_6\|f\|_{H^s}(\|g\|_{H^s}+(2\pi)^n)^s
\eqne
for some $C_6>0$ depending only on $s$ and the constants from some Sobolev embeddings.
\end{lemma}

This allows us to write a second proof of existence and uniqueness of solutions in $\Sigma_s$ for $s>n/2+1$ in the case $s\in \ZZ$.

Fix $u_0 \in H^s$ and $M>0$ and suppose $\eta\in B_M(T)$ for some $T>0$ such that $\eta(t) +\iota$ is volume-preserving for all $t\in[0,T]$. Define $u$ and $v$ via $v=u_0\circ (\eta+\iota)$ and  $u=\PP[(\nabla \eta)^\ast v + v]$. Construct $\eta'$, the iterate of $\eta$ by solving 
\[
\p_t \eta'+(u\cdot\nabla)\eta'=-u, \; \eta'(x,0)=0.
\]
By Lemmas \ref{lemW}, \ref{lemTransport} and \ref{lemComposition} we have
\[
\|\eta'\|_{\Sigma_{s}}\leq\frac{1}{C_4}\left[\exp(C_4C_6(C_3M+1)(M+(2\pi)^n)^s\|u_0\|_{H^s}T)-1\right].
\]
Hence for $T$ small enough, we may assume $\eta'\in B_M(T)$ and since $\nabla \cdot u=0$ we also have that $\eta' +\iota$ is volume preserving.

Now suppose that $\eta_1$, $\eta_2\in B_M(T)$ and let $\eta_1'$, $\eta_2'$ be the respective iterates then 
\[
\|\eta_1'-\eta_2'\|_{\Sigma_0}\leq 2(C_5M+1)(C_3'M+(C_3'M+1)C_{\mathrm{Lip}}) T\|\eta_1-\eta_2\|_{\Sigma_0},
\]
by Lemmas \ref{lemW} and \ref{lemDifferencesTransport}. Here $C_{\mathrm{Lip}}$ is the Lipschitz constant of $u_0$. It follows that, for small enough $T$, this iteration procedure is a contraction on $B_M(T)$ in the $L^2$ norm. Existence and uniqueness of solutions now follows using the same steps as in the previous method.

\section{Conclusions}
Constantin found that $C^{1,\mu}$ initial data gives rise to unique solutions with $C^{1,\mu}$ trajectories for a short time. In contrast, we have seen that for $s>n/2+1$, there exists a local solution which is continuous in time into $H^{s}$ and $C^1$ into $H^{s-1}$ with trajectories in $C^1([0,T]\x{\TT^n})$. This regularity is enough to deduce that such solutions are also solutions of the classical Euler equations.

This paper is partly to prepare the ground for a similar treatment of the Navier--Stokes equations. Once again it is Constantin \cite{Const_ELNS_2001,Const_2003} who has put forward an Eulerian-Lagrangian form for the viscous case. In that formulation diffusive terms appear in the equations for the back-to-labels map and the virtual velocity and in the aforementioned papers some a priori information about that system and its relationship to the classical Navier--Stokes equations are proved. We plan to consider a system for Navier--Stokes with a non-diffusive back to labels map and seek to prove a local existence result analogous to the one exhibited here.

\appendix
\section{Compositions in $H^s$ }\label{Appendix1}
In this appendix we prove Lemma \ref{lemComposition}, which gives bounds on the compositions $H^s$ functions with certain volume-preserving locally $H^s$ functions where $s\in\ZZ$ with $s>\frac{n}{2}$.

To begin with we consider $g_i\in H^s$ and multi indices $\beta_i$ with $|\beta_i|\in[1,s]$ for $i=1,\ldots,\ell$. We call $p\in[1,\infty]$ \textit{admissible} for $(\beta_i)_{1\leq i\leq\ell}$ if there exists a constant $C>0$ independent of $(g_i)_{1\leq i\leq\ell}$ such that
\eqnb\label{app1:eq1}
\left\|\prod_{i=1}^\ell \D^{\beta_i}g_i\right\|_{L^p}\leq C\prod_{i=1}^\ell \|g_i\|_{H^s}.
\eqne
Of course $p$ is admissible if there exist $q_1,\ldots,q_\ell\in[1,\infty)$ such that $H^{s-|\beta_i|}\hookrightarrow L^{q_i}$ for each $i$ and
\[
\sum_{i=1}^\ell \frac{1}{q_i}=\frac{1}{p},
\]
or $p=\infty$ and $q_i=\infty$ for all $i$. We may assume, without loss of generality that there are constants $k_1$ and $k_2$ with $0\leq k_1\leq k_2\leq \ell$ such that 
\[\left\{\begin{array}{l}
s-|\beta_i|\in[0,n/2) \mbox{ for } 1\leq i\leq k_1\\
s-|\beta_i|=n/2 \mbox{ for } k_1+1\leq i\leq k_2\\
s-|\beta_i|>n/2\mbox{ for } k_2+1\leq i\leq \ell\\
\end{array}\right.
\]
So we have 
\[
\left\|\prod_{i=1}^{k_1}D^{\beta_i}g_i\right\|_{L^p}\leq C\prod_{i=1}^{k_1}\|g_i\|_{H^s}
\]
for
\[
\frac{1}{p}\in\left[\sum_{i=1}^{k_1}\frac{n-2(s-|\beta_i|)}{2n},\frac{k_1}{2}\right].
\]
Moreover 
\[
\left\|\prod_{i=k_1+1}^{k_2}D^{\beta_i}g_i\right\|_{L^p}\leq C\prod_{i=k_1+1}^{k_2}\|g_i\|_{H^s}
\]
for $p\in[2,\infty)$. Lastly,
\[
\left\|\prod_{i=k_2+1}^{\ell}D^{\beta_i}g_i\right\|_{L^\infty}\leq C\prod_{i=k_2+1}^{\ell}\|g_i\|_{H^s}.
\]
Combining these observations we see that $p$ is admissible if
\eqnb\label{app1:LpInterval}
\frac{1}{p}\in\left(\sum_{i=1}^{k_1}\frac{n-2(s-|\beta_i|)}{2n},\frac{\ell}{2}\right].
\eqne
or if $k_1=k_2$ then $p$ is still admissible if
\eqnb\label{app1:LpInterval2}
\frac{1}{p}=\sum_{i=1}^{k_1}\frac{n-2(s-|\beta_i|)}{2n},
\eqne
furthermore $p=\infty$ is admissible if $k_1=k_2=0$.

Note that if $p\in[1,\infty]$ is admissable and $f_i:{\TT^n}\to\RR^n$ are linear maps   then we have (rather crudely)
\eqnb\label{app1:LinearAddition}
\left\|\prod_{i=1}^\ell D^{\beta_i}(g_i+f_i)\right\|_{L^p}\leq C\prod_{i=1}^\ell\|g_i\|_{H^s}+\|f_i\|_{\mathrm{op}}(2\pi)^{n/q_i}.
\eqne

In the proof of the lemma below, we will need the fact that if $s>\frac{n}{2}$ and $\sum_{i=1}^\ell|\beta_i|\leq s$ then $p=2$ is admissible for $(\beta_i)_{1\leq i\leq\ell}$.  Furthermore, we will need to show that if $s>n/2+1$ then there exists an admissible $p>\frac{n}{s-\ell}$ and that $p=\infty$ is admissible if $s=\ell>n/2+1$.

For the first claim, note that if $k_1=0$ or $k_1=1$ then $p=2$ is clearly admissible. Otherwise, if $1< k_1\leq\ell$ and $s>n/2$, we have the following calculation:
\eqnb\label{app1:eqClaim1}
\sum_{i=1}^{k_1} n-2(s-|\beta_i|)\leq {k_1} n -2{k_1} s+2s
=(k_1-1)(n-2s)+n < n
\eqne
so $p=2$ is admissible. For the second claim, observe that if $s>n/2+1$ then
\eqnb\label{app1:eqClaim2}
\sum_{i=1}^{k_1} n-2(s-|\beta_i|)< 2\sum_{i=1}^{k_1}|\beta_i| -2k_1\leq 2(s-k_1)-2\sum_{i=k_1+1}^\ell|\beta_i|\leq 2(s-\ell),
\eqne
where the middle inequality uses the assumption that $\sum_{i=1}^\ell|\beta_i|\leq s$. Hence there exists an admissible value $p>\frac{n}{s-\ell}$, if $s-\ell>0$. If $s=\ell $ then necessarily, $|\beta_i|=1$ for $i=1,\ldots,\ell$ hence $p=\infty$ is admissible by \re{app1:LpInterval2}.
\setcounter{TempCounter1}{\value{lem}}
\setcounter{lem}{\value{AppendixLemmaNumber}}
\begin{lemma}\label{app1:lemComposition}
Let $s\in\ZZ$ with $s>\frac{n}{2}+1$ and fix $f, g\in H^s$. Denote the identity map on ${\TT^n}$ by $\iota$. If $g+\iota$ is a volume  preserving map then $f\circ (g+\iota)\in H^s({\TT^n})$ and
\eqnb\label{app1:lemComposition:bound1}
\|f\circ (g+\iota)\|_{H^s}\leq C\|f\|_{H^s}(\|g\|_{H^s}+(2\pi)^n)^s
\eqne
for some $C>0$ depending only on $s$ and the constants from some Sobolev embeddings. 
\end{lemma}
\setcounter{lem}{\value{TempCounter1}}
\begin{proof}
For each $k\in\NN$, consider functions $f_k, g_k\in C^\infty({\TT^n};\RR^n)$ such that $f_k\to f$ in $H^s$ and  $g_k\to g$ in $H^s$.  Without loss of generality we assume that $| |\det\nabla (g_k(x)+x)|-1|<\frac{1}{k+1}$ holds uniformly in $x$.

Now by the chain and Leibniz rules, we see that for a multi-index $\gamma$ with $|\gamma|\leq s$, $D^\gamma(f_k\circ (g_k+\iota))$ is a (weighted) sum with summands of the form
\eqnb\label{app1:lemComposition:derivSummands}
((D^\alpha f_k)\circ (g_k+\iota))\prod_{i=1}^\ell D^{\beta_i}(g_k^{r_i}+x_{r_i}),
\eqne 
where $\ell=|\alpha|\leq|\gamma|$ and $\sum_{i=1}^\ell|\beta_i|=|\gamma|$. Here $g_k^i$ denotes the $i$th vector component of $g_k$. We seek to bound terms of the form \re{app1:lemComposition:derivSummands} in $L^2$ using the preceding observations.

Since $D^\alpha f_k\in H^{s-\ell}$ and $g_k+\iota$ is ``almost volume preserving'' it can be seen that $(D^\alpha f_k)\circ (g_k+\iota)\in L^q$ if 
 \[
\frac{1}{q}\in\left(\frac{1}{2}-\frac{s-\ell}{n},\frac{1}{2}\right]
\] 
with $s-\ell\in(0, n/2]$ or 
\[
\frac{1}{q}=\frac{1}{2}-\frac{s-\ell}{n}
\] when $s-\ell\in(0,n/2)$. Of course, if $s-\ell>n/2$ then $D^\alpha f_k\in L^\infty$. 

To bound \re{app1:lemComposition:derivSummands} in $L^2$ therefore, we need to check that there is an admissible $p$ such that,
\[
\frac{1}{p}\in \left[ 0, \frac{s-\ell}{n}\right).
\]
and that $p=\infty$ is admissible if $s=\ell$. This follows from the claims we proved before the statement of the lemma.

Now we see that 
\[
\|f_k \circ (g_k+\iota)\|_{H^s} \leq C\sqrt{1+1/k} \,\|f_k\|_{H^s}(\|g_k\|_{H^s}+(2\pi)^n)^s
\]
where $C$ depends only on Sobolev embeddings and some combinatorics. Since $f_k$ and $g_k$ converge we may assume that $f_k\circ (g_k+\iota)$ converges weakly in $H^s$. Thus the lemma is proved if we can show that $f_k\circ (g_k+\iota) \to f\circ (g+\iota)$ in $L^2$ for example. This is indeed the case:
\eqnbs
\begin{aligned}
&\|f\circ (g+\iota)-f_k\circ (g_k+\iota)\|_{L^2}\\
&\oneChar\oneChar\oneChar\oneChar \leq \|f\circ (g+\iota) - f\circ (g_k+\iota)\|_{L^2}+\|f\circ (g_k+\iota) -f_k\circ (g_k+\iota)\|_{L^2}\\
&\oneChar\oneChar\oneChar\oneChar\leq C_\mathrm{Lip}\|g-g_k\|_{L^2} + \sqrt{1+1/k}\,\|f - f_k\|_{L^2},
\end{aligned}
\eqnes
where we make use of the fact that $f\in H^s$ is Lipschitz since $s>n/2+1$ and denote by $C_\mathrm{Lip}$ the Lipschitz constant of $f$.
\end{proof}

\end{document}